\documentclass[12pt]{article}
\usepackage{graphicx}
\usepackage{amsmath}
\usepackage{amsfonts}
\usepackage{amssymb}
\usepackage{float}
\providecommand{\U}[1]{\protect\rule{.1in}{.1in}}
\providecommand{\U}[1]{\protect\rule{.1in}{.1in}}
\newtheorem{theorem}{Theorem}

\newtheorem{proposition}[theorem]{Proposition}

\newenvironment{proof}[1][Proof]{\noindent\textbf{#1.} }{\ \rule{0.5em}{0.5em}}
\begin{document}
$\ $

\vspace{-1.cm}

\begin{center}
{\Large \textbf{The Six Cylinders Problem: $\mathbb{D}_{3}$-symmetry\\[.2em]
Approach}}

\vspace{.3cm} {\large \textbf{Oleg Ogievetsky$^{\diamond\,\ast}$\footnote{Also
at Lebedev Institute, Moscow, Russia.} and Senya Shlosman$^{^{\diamond}
\,\dag\,\ddagger}$}}

\vskip .3cm $^{\diamond}$Aix Marseille Universit\'{e}, Universit\'{e} de
Toulon, \newline CNRS, CPT UMR 7332, 13288, Marseille, France

\vskip .05cm $^{\dag}$Inst. of the Information Transmission Problems, RAS,
Moscow, Russia

\vskip .05cm $^{\ddagger}$ Skolkovo Institute of Science and Technology,
Moscow, Russia

\vskip .05cm $^{\ast}${Kazan Federal University, Kremlevskaya 17, Kazan
420008, Russia}
\end{center}

\vskip 2cm
\begin{minipage}[t]{.9\textwidth}{\small
\centerline{\bf Abstract}
\vskip .6cm
Motivated by a question of W. Kuperberg, we study the 18-dimensio\-nal manifold
of configurations of 6 non-intersecting infinite cylinders of radius $r,$ all
touching the unit ball in $\mathbb{R}^{3}.$ We find
a configuration with
\[ r=\frac{1}{8}\left(  3+\sqrt{33}\right)  \approx1.093070331\ .\]
We believe that this value is the maximum possible.
}\end{minipage}

\newpage
\section{Introduction}
$\ $
\vskip .1cm
The question: -
%\begin{center}
How many non-intersecting unit right circular (open) cylinders of infinite length can touch a unit ball? 
%\end{center}
\noindent - was asked by W. Kuperberg, \cite{K}. 

\vskip .3cm
Kuperberg presented several arrangements of 6
non-intersecting unit cylinders touching the unit ball; it is difficult
to imagine that 7 unit cylinders of infinite length can do it, though no proof of this statement
is known; see \cite{HS} for the proof that 8 unit non-intersecting cylinders of infinite length cannot touch the unit ball. 

\vskip .3cm
At first glance one can even think that 6 non-intersecting cylinders
of radius $r>1$ cannot touch the unit ball. This, however, is not the case,
and an example was presented by M. Firsching in his thesis, \cite{F}. In this
example the radius $r$ equals $1.049659.$ This example was obtained by a
numerical exploration of the corresponding 18-dimensional configuration manifold.

\vskip.3cm 
The situation thus is somewhat similar to the case of 12 unit
balls touching the central unit ball. There one can similarly ask whether 13
unit balls can do it (the answer is negative, \cite{SW}), or whether 12 balls
of bigger radius $r>1$ can touch the central unit ball. The answer to the
latter question is positive: it is known that 12 balls of radius 
$$r=\left( \sqrt{\frac{5+\sqrt{5}}{2}}-1\right)  ^{-1}\approx1.10851\ ,$$ 
positioned at the 12 vertices of the icosahedron with edge $2r,$ touch the central 
unit ball.

\vskip.3cm 
This fact makes it plausible that the two very symmetric configurations
of 12 unit balls touching the central unit one -- the \text{FCC} (Face
Centered Cubic) and the \text{HCP} (Hexagonal Closed Packed) configurations
(see Figures \ref{FCC} and \ref{HCP} for explanation) -- can be
\textit{unlocked} by rolling the 12 balls over the central one to a
configuration where none of the 12 balls touch each other. 
This is indeed correct; see Chapter VII, \S$\,$2 in \cite{T} and \S$\,$ 8.4 in \cite{C} for the configuration FCC, and  
\cite{KKLS} for the configuration HCP.

\vspace{-0.4cm}

\begin{figure}[H]
\centering
\includegraphics[scale=0.26]{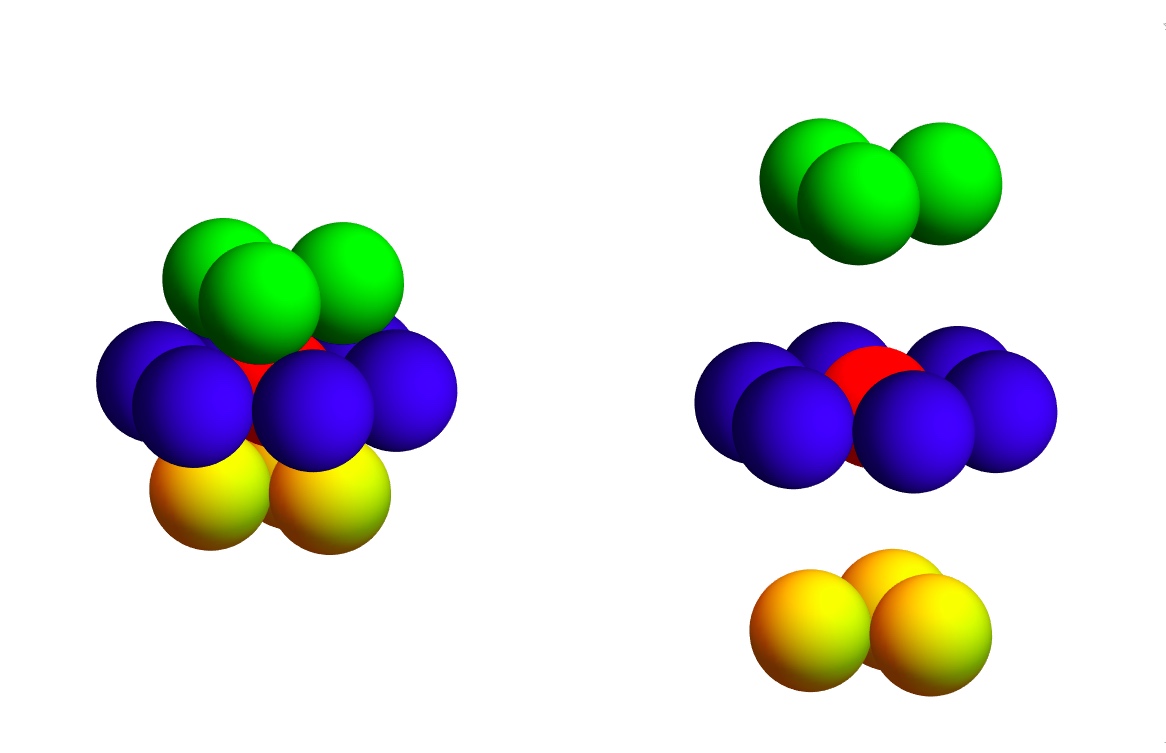} \caption{FCC configuration (left) and
its layers (right)}
\label{FCC}
\end{figure}

\vspace{-0.4cm} \begin{figure}[H]
\centering
\includegraphics[scale=0.32]{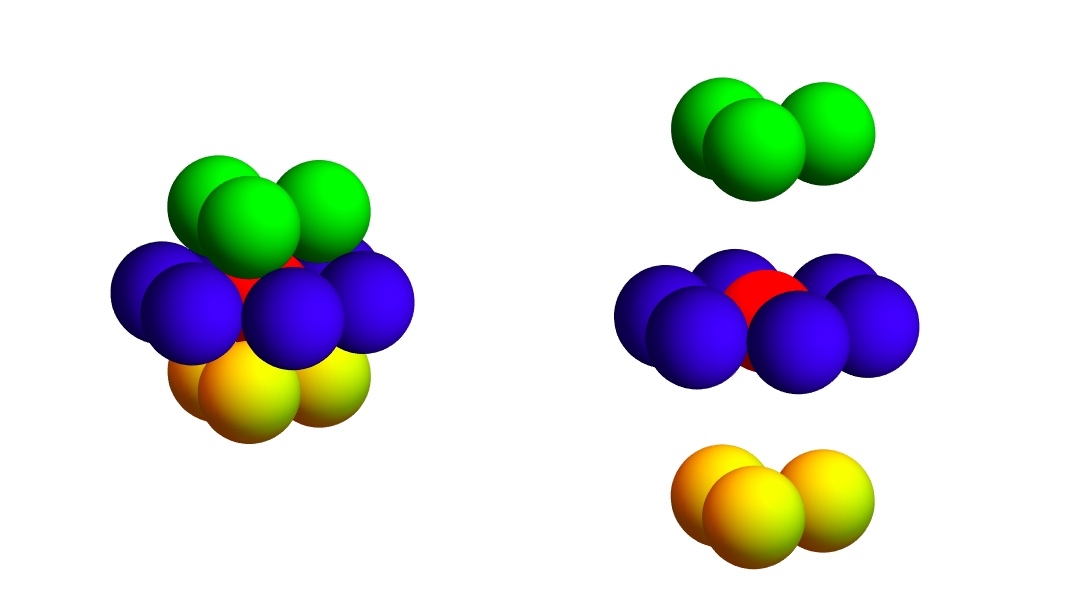} \caption{HCP configuration (left) and
its layers (right)}
\label{HCP}
\end{figure}

The precise meaning of  \textit{unlocking} is the following. Let $G$ be a
collection of solid bodies, $G=\left\{  \Lambda_{1},...\Lambda_{k}\right\}  ,$
where each $\Lambda_{i}$ touches the unit central ball, while some distances
between bodies of $G$ are zero. We say that $G$ can be unlocked if there
exists a continuous deformation $G\left(  t\right)  ,$ $t\geq0,$ of $G$ (i.e.
$G\left(  0\right)  =G$), such that for any $t>0$ all the distances between
the members in the configuration $G\left(  t\right)  $ are positive, while
each $\Lambda_{i}$ touches the central ball while moving.

\vskip.2cm In the present paper we address a similar question -- of unlocking
the configuration of six unit parallel (right circular) cylinders, touching
the central unit ball. We denote this configuration by $C_{6}$, see Figure
\ref{confC6}.

\vspace{0.4cm} \begin{figure}[ht]
\centering
\includegraphics[scale=0.12]{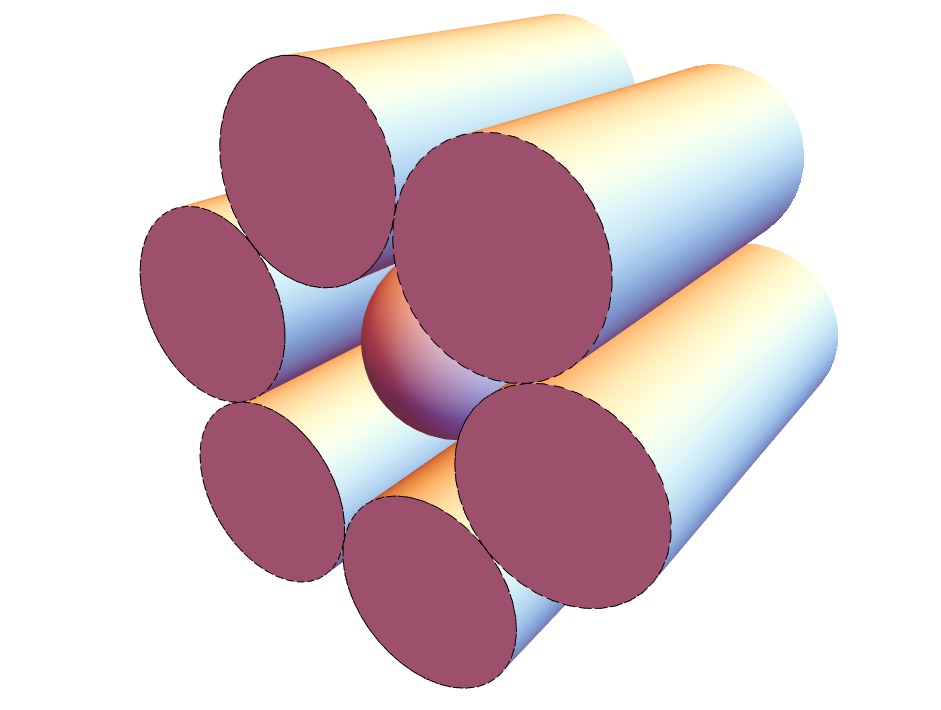} \caption{Configuration
$C_{6}$}
\label{confC6}
\end{figure}

The configuration $C_{6}$ is not rigid. Indeed, let $H\subset\mathbb{R}^{3}$
be a half-space, containing three cylinders of $C_{6},$ and $h$ be the normal
vector to the plane $\partial H.$ Then one can rotate the three cylinders
about $h,$ keeping the remaining three intact, see Figure \ref{nonrigconfC6}.

\vspace{0.4cm} \begin{figure}[ht]
\centering
\includegraphics[scale=0.23]{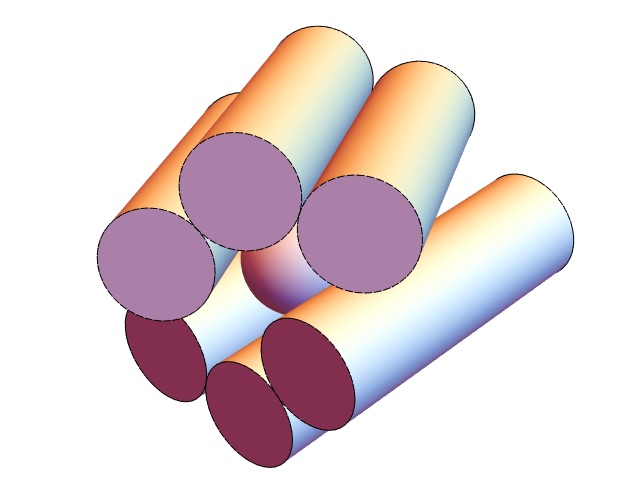}
\caption{Non-rigidity of $C_{6}$}
\label{nonrigconfC6}
\end{figure}

So our configuration $C_{6}$ is movable, but this is not yet the unlocking,
since some distances stay zero. We will demonstrate that the configuration
$C_{6}$ is indeed unlockable. Namely, we will present its continuous
deformation $C_{6}\left(  t\right)  $, along which quite a spacing opens
between the cylinders, so at some value of $t$ it becomes possible to arrange
6 non-intersecting cylinders of radius
\begin{equation}
r_{\mathfrak{m}}=\frac{1}{8}\left(  3+\sqrt{33}\right)  \approx1.093070331.
\label{30}
\end{equation}

We believe that our configuration of 6 cylinders with radius $r_{\mathfrak{m}}$
is in fact optimal. In a forthcoming publication \cite{OS} we are going to
show that our cylinder arrangement with value
$r_{\mathfrak{m}}$, is a local maximum, i.e. any small perturbation
of our configuration decreases the corresponding radius.

\vskip .2cm
The search of the maximum radius $r$ is equivalent to finding a point in a
certain 18-dimensional manifold $M^{6},$ see the definition $\left(
\ref{25}\right)  $ below, where the minimum of 15 mutual distances attains its
maximum value. Guided by our belief that the optimal configuration should
possess nice symmetries, we restricted our search to a certain 3-dimensional
submanifold $\mathcal{C}^{3}=C_{6}\left(  \varphi,\delta,\varkappa\right)  $
of $M^{6}$, see the definition $\left(  \ref{defconf}\right)  $ below,
consisting of the fixed points of the action of the group $\mathbb{D}
_{3}\subset SO\left(  3\right)  $ on $M^{6},$ i.e. by $\mathbb{D}_{3}
$-symmetric configurations. On $\mathcal{C}^{3}$, only 4 of 15 distances are
different, and only 3 of them are relevant. Our next reduction comes from the
observation that the situation when three `nice' functions $g_{1},g_{2},g_{3}$
on a three-dimensional manifold $N$ coincide on a smooth curve $\gamma$ is a
general position situation, as the dimension counting immediately shows. In
such a case the point $x_{\mathfrak{m}}\in N$ at which the $\max$ of the function
$\min_{i}\left\{  g_{i}\left(  x\right)  \right\}  $ is attained, belongs to
$\gamma.$ It so happens that our case (with $g_{1},g_{2},g_{3}$ being the
three relevant distances) falls into it, with $\gamma=C_{6}\left(
\varphi,\delta\left(  \varphi\right)  ,\varkappa\left(  \varphi\right)
\right)  ,$ for certain functions $\delta\left(  \varphi\right)
,\varkappa\left(  \varphi\right)  $. What is left then is the study of a
single function $g_{i}{|}_{\gamma}$ of one variable. We were able to explicitly describe this curve
$\gamma\subset M^{6}$ and to compute the maximum value $r_{\mathfrak{m}}$ of
the function $r$ on it. It gives a lower bound for the maximum radius $r$ possible.

\vskip.2cm We also analyze the generalized situation, with $2n$ cylinders
instead of 6. We show that it can be unlocked for $n>2$ along our curve. For
$n=2$ the configuration is not rigid but it is not unlockable along our curve.
However, we conjecture that all possible configurations of four cylinders
belong to the curve.

\vskip.2cm The description of our 18-dimensional manifold $M^{6}$ and the
choice of coordinates there are given in the next section. Section 3 contains the
definition of the submanifold $\mathcal{C}^{3}\subset M^{6}$ and the
formulation of our main result. The optimization problem on $\mathcal{C}^{3}$
is solved in Sections 4 and 5, thus proving our main theorem. In Section 6 we
consider the problem of $n$ equal cylinders touching the unit ball.
The last Section 7 contains our conclusions.

\vskip.2cm 
We finish the introduction by the brief history of how the present
paper was evolving. Our first goal was to convince ourselves that the
configuration $C_{6}$ is \textit{infinitesimally} unlockable (see Proposition
3). Next, we were trying to analyze the humongous trigonometric formulas for
the functions $g_{i}\left(  x\right) $, and we used both Wolfram Mathematica
\cite{W} and the analog machinery:

\begin{figure}[th]
\centering
\includegraphics[scale=0.06]{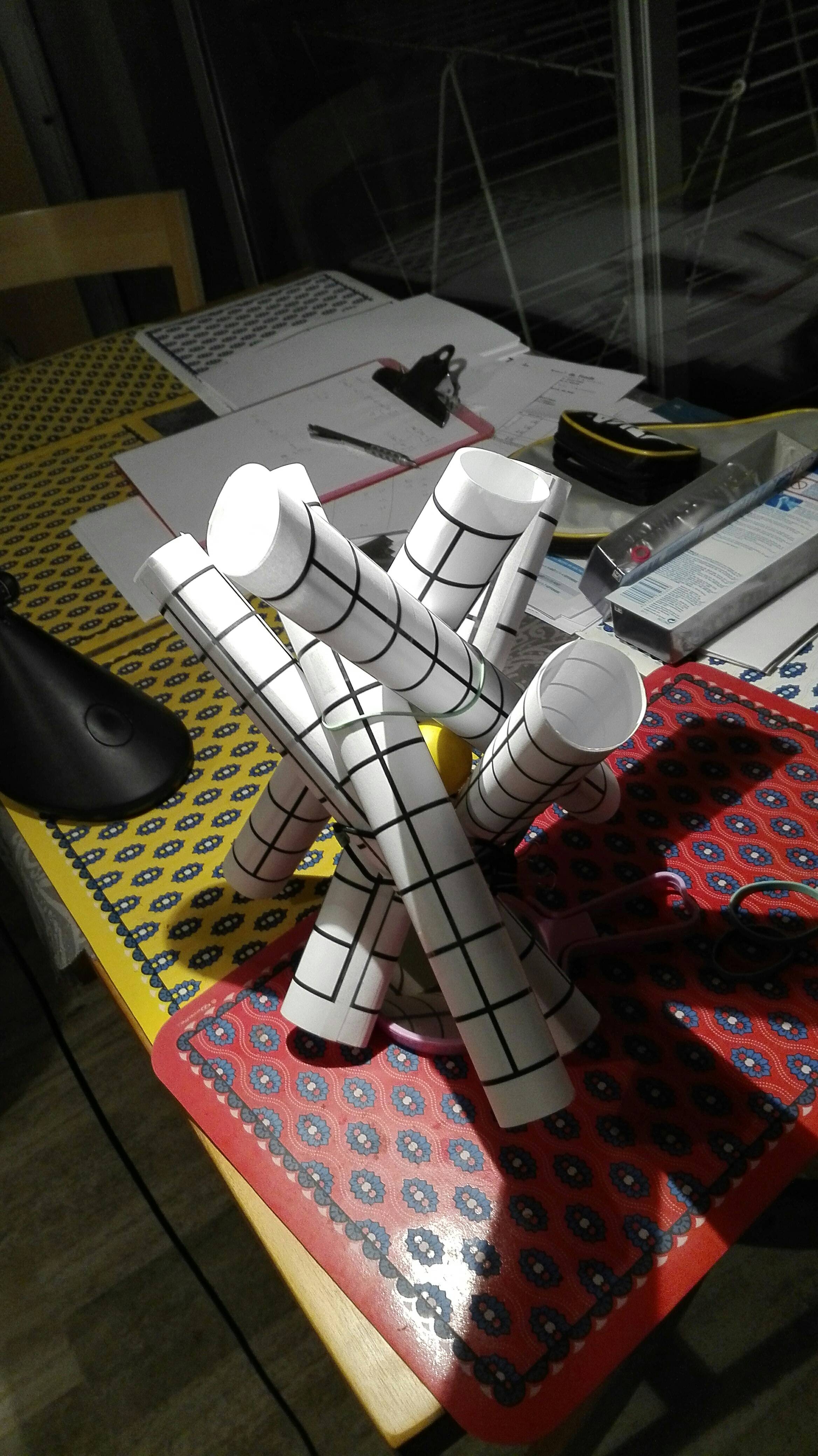}
\caption{The analog computer. The yellow ball is visible in the center.}
\end{figure}
\noindent to solve the minimax problem numerically. We got an estimate
$1.09$ for $r_{\mathfrak{m}}.$ The last phase came with the realization that
it is possible to pass from trigonometric expressions to algebraic ones,  such
that our minimax problem
becomes \textit{`integrable'}, i.e. can be solved explicitly. In our view this is quite a surprising feature of the six cylinder
problem, which is beyond our initial expectations. Probably, this points to some
hidden symmetry features of the problem.

\section{The configuration manifold}

Let $\mathbb{S}^{2}\subset\mathbb{R}^{3}$ be the unit sphere, centered at the
origin. For every $x\in\mathbb{S}^{2}$ by $TL_{x}$ we denote the set of all
(unoriented) tangent lines to $\mathbb{S}^{2}$ at $x.$ 
The manifold of tangent lines to $\mathbb{S}^{2}$ we denote by $M$, and
we represent a point in $M$ by a pair $\left(  x,\tau\right)  $, where $\tau$
is a unit tangent vector to $\mathbb{S}^{2}$ at $x,$ though such a pair is not
unique: the pair $\left(  x,-\tau\right)  $ is the same point in $M.$ We shall
use the following coordinates on $M$. Let $\mathbf{x,y,z}$ be the standard
coordinate axes in $\mathbb{R}^{3}$. Let $R_{\mathbf{x}}^{\alpha
},R_{\mathbf{y}}^{\alpha}$ and $R_{\mathbf{z}}^{\alpha}$ be the
counterclockwise rotations about these axes by an angle $\alpha$, viewed from
the tips of axes. We call the point $\mathsf{N}=\left(  0,0,1\right)  $ the
North pole, and $\mathsf{S}=\left(  0,0,-1\right)  $ -- the South pole. By
\textit{meridians} we mean geodesics on $\mathbb{S}^{2}$ joining the North
pole to the South pole. The meridian in the plane $\mathbf{xz}$ with positive
$\mathbf{x}$ coordinates will be called Greenwich. The angle $\varphi$ will
denote the latitude on $\mathbb{S}^{2},$ $\varphi\in\left[  -\frac{\pi}
{2},\frac{\pi}{2}\right]  ,$ and the angle $\varkappa\in\lbrack0,2\pi)$ -- the
longitude, so that Greenwich corresponds to $\varkappa=0.$ Every point
$x\in\mathbb{S}^{2}$ can be written as $x=\left(  \varphi_{x},\varkappa
_{x}\right)  .$ Finally, for each $x\in\mathbb{S}^{2}$, we denote by
$R_{x}^{\alpha}$ the rotation by the angle $\alpha$ about the axis joining
$\left(  0,0,0\right)  $ to $x,$ counterclockwise if viewed from its tip, and
by $\left(  x,\uparrow\right)  $ we denote the pair $\left(  x,\tau
_{x}\right)  ,$ $x\neq\mathsf{N,S,}$ where the vector $\tau_{x}$ points to the
North. We also abbreviate the notation $\left(  x,R_{x}^{\alpha}
\uparrow\right)  $ to $\left(  x,\uparrow_{\alpha}\right)  $.

\vskip .2cm
Let $u=\left(  x^{\prime},\tau^{\prime}\right)  ,$ $v=\left(  x^{\prime\prime
},\tau^{\prime\prime}\right)  $ be two lines in $M$. We denote by $d_{uv}$ the
distance between
$u$ and $v$; clearly $d_{uv}=0$ iff $u\cap v\neq\varnothing.$ If the lines
$u,v$ are not parallel then the square of $d_{uv}$ is given by the formula
\[
d_{uv}^{2}=\frac{\det^{2}[\tau^{\prime},\tau^{\prime\prime},x^{\prime\prime
}-x^{\prime}]}{1-(\tau^{\prime},\tau^{\prime\prime})^{2}}\ ,
\]
where $(\ast,\ast)$ is the scalar product. For the future use we note that if
$d_{uv}=d>0,$
then the cylinders $C_{u}\left(  r\right)  $ and $C_{v}\left(  r\right)  ,$
touching $\mathbb{S}^{2}$ at $x^{\prime},x^{\prime\prime},$ having directions
$\tau^{\prime},\tau^{\prime\prime},$ and radius $r,$ touch each other iff
\begin{equation}
r=\frac{d}{2-d}. \label{11}
\end{equation}
Indeed, if the cylinders touch each other, we have the proportion:
\begin{equation}
\frac{d}{1}=\frac{2r}{1+r}. \label{10}
\end{equation}

We denote by $M^{6}$ the manifold of 6-tuples
\begin{equation}
\mathbf{m}=\left\{  u_{1},...,u_{6}:u_{i}\in M,i=1,...,6\right\}  . \label{25}
\end{equation}
Our interest is in the function
\[
D\left(  \mathbf{m}\right)  =\min_{1\leq i<j\leq6}d_{u_{i}u_{j}}.
\]
We are especially interested in knowing its maximum, since it defines, via
$\left(  \ref{11}\right)  ,$ the maximum radius of
6 non-intersecting equal cylinders touching the unit ball.

\vskip .2cm
The generators of the cylinders in $C_{6}$ touching the ball define a point in
$M^{6}$, shown on Figure \ref{confC6tan}. We denote it by the same symbol
$C_{6}$. Note that
$D\left(  C_{6}\right)  =1$. \vspace{1cm} \begin{figure}[th]
\vspace{-1.2cm} \centering
\includegraphics[scale=0.16]{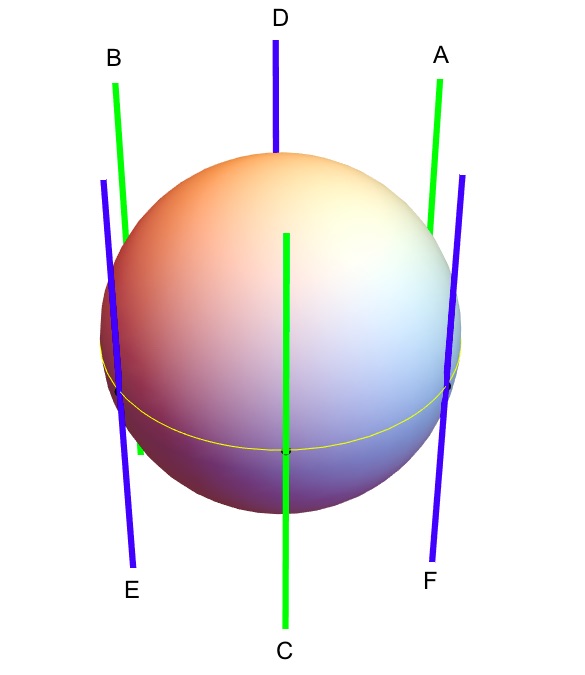}
\caption{Configuration $C_{6}$ of tangent lines}
\label{confC6tan}
\end{figure}

\section{Points $\mathbf{m}\in M^{6}$ with high $D\left(  \mathbf{m}\right)  $
value}

Here we describe the
`good' configurations $\mathbf{m}$ with
high values of the function $D\left(  \mathbf{m}\right)  .$ We obtain them by
deforming the configuration $C_{6}$ which in our notation can
be written as
\[
\begin{array}
[c]{ll}
C_{6} & \equiv C_{6}\left(  0,0,0\right)  =\left\{  \left[  \left(
0,\frac{\pi}{6}\right)  ,\uparrow\right]  ,\left[  \left(  0,\frac{\pi}
{2}\right)  ,\uparrow\right]  ,\left[  \left(  0,\frac{5\pi}{6}\right)
,\uparrow\right]  ,\right. \\[.8em]
& \hspace{3cm} \left.  \left[  \left(  0,\frac{7\pi}{6}\right)  ,\uparrow
\right]  ,\left[  \left(  0,\frac{3\pi}{2}\right)  ,\uparrow\right]  ,\left[
\left(  0,\frac{11\pi}{6}\right)  ,\uparrow\right]  \right\}  .
\end{array}
\]

Namely, we will explore the 6-tuples $C_{6}\left(  \varphi,\delta
,\varkappa\right)  $, of the form
\begin{equation}
\label{defconf}
\begin{array}
[c]{ll}
& C_{6}\left(  \varphi,\delta,\varkappa\right)  = \left\{  A=\left[  \left(
\varphi,\frac{\pi}{6}-\varkappa\right)  ,\uparrow_{\delta}\right]  ,D=\left[
\left(  -\varphi,\frac{\pi}{2} +\varkappa\right)  ,\uparrow_{\delta}\right]
,\right. \\[.8em]
& \hspace{2.74cm} B=\left[  \left(  \varphi,\frac{5\pi}{6}-\varkappa\right)
,\uparrow_{\delta}\right]  ,E=\left[  \left(  -\varphi,\frac{7\pi}
{6}+\varkappa\right)  ,\uparrow_{\delta}\right]  ,\\[.8em]
& \hspace{2.74cm} \left.  C=\left[  \left(  \varphi,\frac{3\pi}{2}
-\varkappa\right)  ,\uparrow_{\delta}\right]  ,F=\left[  \left(
-\varphi,\frac{11\pi} {6}+\varkappa\right)  ,\uparrow_{\delta}\right]
\right\}  .
\end{array}
\end{equation}
In words, the three points $\left[  \left(  0,\frac{\pi}{6}\right)
,\uparrow\right]  ,\left[  \left(  0,\frac{5\pi}{6}\right)  ,\uparrow\right]
$ and $\left[  \left(  0,\frac{3\pi}{2}\right)  ,\uparrow\right]  $ go upward
by $\varphi,$ then `horizontally' by $-\varkappa,$ and then the three vectors
$\uparrow$ are rotated by $\delta,$ while the three remaining points go
downward by $\varphi$, then `horizontally' by $\varkappa,$ and, finally, the
three vectors $\uparrow$ are rotated by $\delta$.

\vskip .2cm
For all $\varphi,\delta,\varkappa$ these configurations possess $\mathbb{D}
_{3}\equiv\mathbb{Z}_{3}\times\mathbb{Z}_{2}$ symmetry. The group
$\mathbb{D}_{3}$ is generated by the rotations $R_{\mathbf{z}}^{120^{\circ}}$
and $R_{\mathbf{x}}^{180^{\circ}}.$ We denote by $\mathcal{C}^{3}\in M^{6}$
the 3-dimensional submanifold formed by 6-tuples $\left(  \ref{defconf}
\right)  $.

\vskip .2cm
We claim that there exists a curve $\gamma$ in the manifold $\mathcal{C}^{3}
$,
\begin{equation}
\gamma(\varphi)=C_{6}\bigl(\varphi,\delta\left(  \varphi\right)
,\varkappa\left(  \varphi\right)  \bigr)\ ,\ \varphi\in\left[  0;\frac{\pi}
{2}\right]  \ , \label{curvegamma}
\end{equation}
which starts at $C_{6}\left(  0,0,0\right)  $ for $\varphi=0$,
\begin{equation}
\gamma(0)=C_{6}\left(  0,0,0\right)  \ , \label{curvegamma2}
\end{equation}
such that the function $D\bigl(\gamma(\varphi)\bigr)$
is unimodal on $\gamma,$ with maximum value
$\sqrt{\frac{12}{11}}$,
which corresponds to the value
$r_{\mathfrak{m}}$, given in $\left(  \ref{30}\right)  $, of the radii of the
touching cylinders. This is summarized in our main result below. Its proof
constitutes a part of Section \ref{secsolving}.

\begin{theorem}
\label{Main} The configuration $C_{6}\left(  0,0,0\right)  $ can be
unlocked. Moreover,

\vskip .2cm
\textbf{i. }There is a continuous curve $\gamma$, see (\ref{curvegamma}) and
(\ref{curvegamma2}),
on which the function $D\bigl(  \gamma(\varphi)\bigr) $
increases for $\varphi\in\left[  0,\varphi_{\mathfrak{m}}\right]  $ and decreases for
$\varphi>\varphi_{\mathfrak{m}},$ with $\varphi_{\mathfrak{m}}=\arcsin
\sqrt{\frac{3}{11}}.$ The explicit description of $\gamma$ is given in
(\ref{traj1})-(\ref{traj3}).

\vskip .3cm
\textbf{ii.} At the point $\varphi_{\mathfrak{m}},\delta_{\mathfrak{m}}
=\delta\left(  \varphi_{\mathfrak{m}}\right)  ,\varkappa_{\mathfrak{m}
}=\varkappa\left(  \varphi_{\mathfrak{m}}\right)  $ we have
\[
D\Bigl(  C_{6}\left(  \varphi_{\mathfrak{m}},\delta_{\mathfrak{m}}
,\varkappa_{\mathfrak{m}}\right)  \Bigr)  =\sqrt{\frac{12}{11}}\ ,
\]
so the radii of the corresponding cylinders are equal to
\[
r_{\mathfrak{m}}=\frac{1}{8}\left(  3+\sqrt{33}\right)  .
\]

\end{theorem}

We stress again that the existence of analytic expression for the curve
$\gamma$ comes beyond expectations, and seems quite surprising.

\vskip .2cm
The record configuration is shown on Figures \ref{record1}, \ref{record2}
and \ref{record3}.

\newpage

\begin{figure}[h!]
\centering
\includegraphics[scale=0.22]{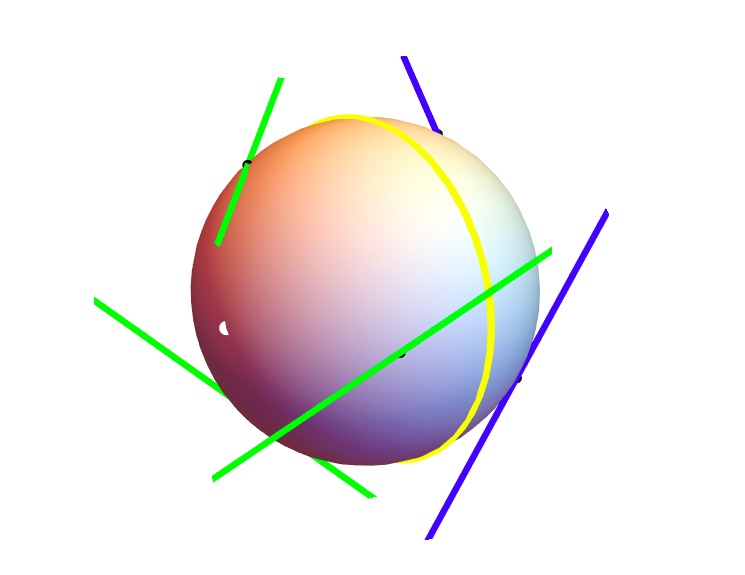}
\caption{Record configuration, side view, the equator is yellow, the north
pole is white\label{record1}}
\includegraphics[scale=0.2]{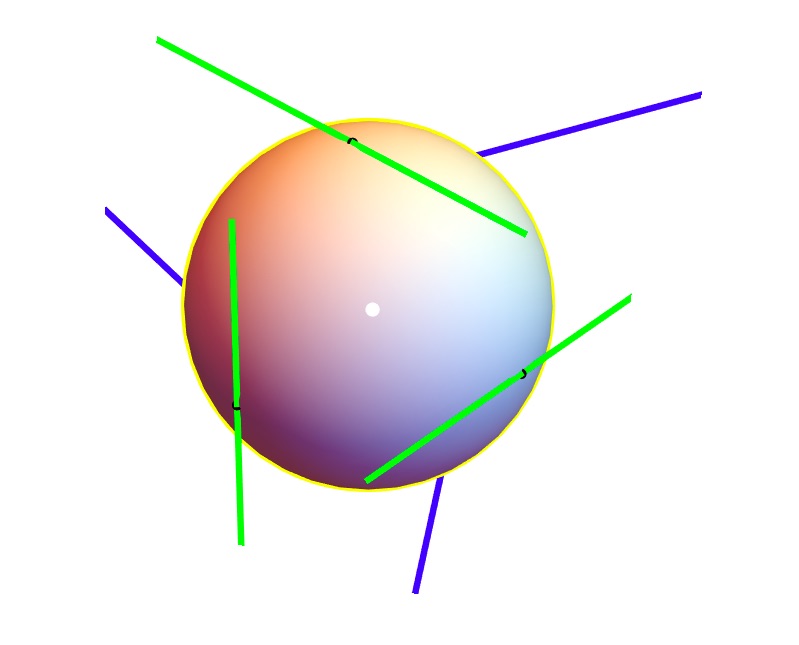}
\caption{ Record configuration again, three upper tangency points shown\label{record2}}
\includegraphics[scale=0.2]{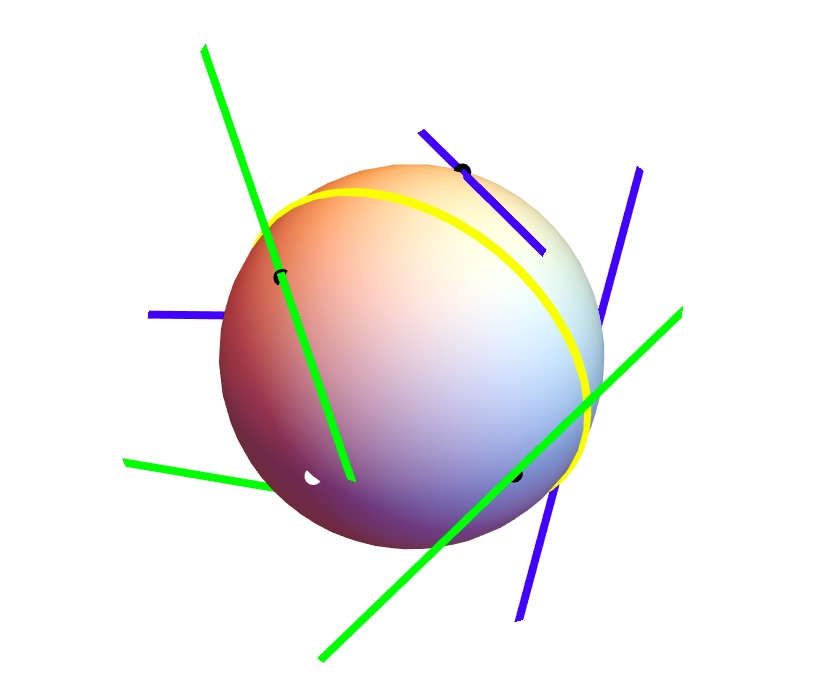}
\caption{Record configuration once more, two upper and one lower tangency points shown\label{record3}}
\end{figure}

\newpage

\section{Formulas for $\mathbb{D}_{3}$-symmetric configurations \label{F}}

Now we present explicit formulas for exploring the manifold
$\mathcal{C}^{3}.$ Because of the $\mathbb{D}_{3}$-symmetry, $d_{AB}
=d_{BC}=d_{CA}=d_{DE}=d_{EF}=d_{FD},$ so we need only one of these, which is
given by
\[
d_{AB}^{2}=\frac{48\sin^{2}(\delta)\cos^{2}(\delta)\cos^{4}(\varphi)}{\left(
6\cos^{2}(\delta)\cos(2\varphi)+3\cos(2\delta)+7\right)  \left(  \cos
^{2}(\delta)\sin^{2}(\varphi)+\sin^{2}(\delta)\right)  }\ ,
\]
which, naturally, does not depend on $\varkappa.$ Also, $d_{AD}=d_{BE}
=d_{CF},$ with
\[
d_{AD}^{2}=\frac{\mu_{AD}^{2}}{4(1-\nu_{AD}^{2})}\ ,\ \text{where}
\]
\[
\begin{array}
[c]{rcl}
\mu_{AD} & = & \sin(2\delta)\Bigl(2\cos^{2}(\varphi)-(\cos(2\varphi
)-3)\sin\left(  2\text{$\varkappa$}-\frac{\pi}{6}\right)  \Bigr)\\[1em]
& + & 4\cos(2\delta)\sin(\varphi)\cos\left(  2\text{$\varkappa$}-\frac{\pi}
{6}\right)  \ ,
\end{array}
\]
\[
\begin{array}
[c]{rcl}
\nu_{AD} & = & \sin\left(  2\text{$\varkappa$}-\frac{\pi}{6}\right)
\Bigl(\sin^{2}(\delta)-\cos^{2}(\delta)\sin^{2}(\varphi)\Bigr)\\[1em]
& + & \sin(2\delta)\sin(\varphi)\cos\left(  2\text{$\varkappa$}-\frac{\pi}
{6}\right)  -\cos^{2}(\delta)\cos^{2}(\varphi)\ .
\end{array}
\]
The third triplet of functions is $d_{BD}=d_{CE}=d_{AF},$ with
\[
d_{BD}^{2}=\frac{\mu_{BD}^{2}}{4(1-\nu_{BD}^{2})}\ ,\ \text{where}
\]
\[
\begin{array}
[c]{rcl}
\mu_{BD} & = & \sin(2\delta)\Bigl(2\cos^{2}(\varphi)-(\cos(2\varphi
)-3)\sin\left(  2\text{$\varkappa$}-\frac{5\pi}{6}\right)  \Bigr)\\[1em]
& + & 4\cos(2\delta)\sin(\varphi)\cos\left(  2\text{$\varkappa$}-\frac{5\pi
}{6}\right)  \ ,
\end{array}
\]
\[
\begin{array}
[c]{rcl}
\nu_{BD} & = & \sin\left(  2\text{$\varkappa$}-\frac{5\pi}{6}\right)
\Bigl(\sin^{2}(\delta)-\cos^{2}(\delta)\sin^{2}(\varphi)\Bigl)\\[1em]
& + & \sin(2\delta)\sin(\varphi)\cos\left(  2\text{$\varkappa$}-\frac{5\pi}
{6}\right)  -\cos^{2}(\delta)\cos^{2}(\varphi).
\end{array}
\]
The last triplet is $d_{AE}=d_{BF}=d_{CD},$ with
\[
d_{AE}^{2}=\frac{\mu_{AE}^{2}}{\nu_{AE}}\ ,\ \text{where}
\]
\[
\mu_{AE}=2\Bigl(\cos\left(  \delta\right)  \cos\left(  \varkappa\right)
-\sin\left(  \delta\right)  \sin\left(  \varphi\right)  \sin\left(
\varkappa\right)  \Bigr),
\]
\[
\nu_{AE}=\cos^{2}(\delta)  \cos^{2}\left(  \varphi\right)
+\left(  \sin (\delta ) \right) \sin\left(  \varkappa\right)  -\cos\left(
\delta\right)  \sin\left(  \varphi\right)  \cos\left(  \varkappa\right)
\Bigr)^{2}.
\]
The derivation of the above formulas is straightforward, though tedious.
It is difficult to explore these formulas directly. However, there is a
suitable choice of variables,
such that instead of ratios of trigonometric polynomials involving various
$\sin$-s and $\cos$-s of various angles,
the square of each distance becomes a rational function.

\begin{proposition}
Let
\begin{equation}
S=\sin(\varphi)\ ,\ T=\tan(\delta)\ , \label{d6c3}
\end{equation}
and
\[
U=\tan(\varkappa-\frac{\pi}{6})\ ,\ \bar{U}=-\tan(\varkappa+\frac{\pi}{6})\ .
\]
Then
\begin{equation}
d_{AB}^{2}=\frac{12T^{2}\left(  1-S^{2}\right)  ^{2}}{(4-3S^{2}+T^{2}
)(S^{2}+T^{2})}\ , \label{31}
\end{equation}
\begin{equation}
d_{AD}^{2}=\frac{4\left(  TS+U\right)  ^{2}}{1+U^{2}+T^{2}-S^{2}+2STU}\ ,
\label{32}
\end{equation}
\begin{equation}
d_{BD}^{2}=\frac{4\left(  -TS+\bar{U}\right)  ^{2}}{1+\bar{U}^{2}+T^{2}
-S^{2}-2ST\bar{U}}\ . \label{33}
\end{equation}

\end{proposition}

Since $(\varkappa+\frac{\pi}{6})=(\varkappa-\frac{\pi}{6})+\frac{\pi}{3}$, we
have $\left(  \text{via }\tan(\beta_{1}+\beta_{2})=\frac{\tan(\beta_{1}
)+\tan(\beta_{2})}{1-\tan(\beta_{1})\tan(\beta_{2})}\right)  $:
\begin{equation}
\bar{U}=-\frac{U+\sqrt{3}}{1-\sqrt{3}U}\ \ \text{or}\ \ -\sqrt{3}U\bar
{U}+U+\bar{U}+\sqrt{3}=0\ . \label{d6c13}
\end{equation}
The proof of the proposition is elementary: one just needs to check various
trigonometric identities. Yet to find the right choice of variables, allowing
further analysis,
was the longest part of the present work, involving lengthy and painful computations.

\section{Solving $d_{AB}^{2}=d_{BD}^{2}=d_{AD}^{2}.$}

\label{secsolving}

In this section we will write the functions $d_{AB}^{2},d_{BD}^{2},d_{AD}
^{2},$ given by relations $\left(  \ref{31}-\ref{33}\right)  ,$ on the curve
$d_{AB}^{2}=d_{BD}^{2}=d_{AD}^{2}$ as functions of one parameter, and then
will maximize them. Also we will use, of course, the relation $\left(
\ref{d6c13}\right)  .$

\vskip .2cm
The angle $\varphi$ is positive in the region of our interest, so the factors
in the denominator of $d_{AB}^{2}$ do not vanish. The denominator of
$d_{AD}^{2}$ (similarly for $d_{BD}^{2}$) can be written in the form
$(U+ST)^{2} +(1+T^{2})(1-S^{2})$ so it does not vanish either.

\vskip .2cm
The equality $d_{AD}^{2}=d_{BD}^{2}$ gives
\begin{equation}
(T^{2}+1)(S^{2}-1)(U+\bar{U})(U-\bar{U}+2ST)=0\ . \label{d6c14}
\end{equation}
The factor $(T^{2}+1)$ is non-zero, the factor $(S^{2}-1)$ is non-zero at any
point in our $(\varphi,\varkappa,\delta)$-space except the initial point
$\varphi=0$, the factor $(U+\bar{U})$ is non-zero (see (\ref{d6c13})) so we
conclude
\begin{equation}
U-\bar{U}+2ST=0\ . \label{d6c14b}
\end{equation}
Jointly, eqs. (\ref{d6c13}) and (\ref{d6c14b}) imply
\[
K_{1}\equiv-\sqrt{3}U^{2}+2U\left(  1-\sqrt{3}ST\right)  +2ST+\sqrt{3}=0.
\]
The equality $d_{AB}^{2}=d_{AD}^{2}$ leads to 
\[
(T^{2}+1)\left[  (-4S^{2}+3S^{4}-T^{2})(U+ST)^{2}-3T^{2}(S^{2}-1)^{3}\right]
=0.
\]
The factor $(T^{2}+1)$ does not vanish, so we obtain:
\[
K_{2}\equiv(-4S^{2}+3S^{4}-T^{2})(U+ST)^{2}-3T^{2}(S^{2}-1)^{3}=0.
\]
Therefore
\[
(-4S^{2}+3S^{4}-T^{2})K_{1}-\sqrt{3}K_{2}=0,
\]
which reads
\[
-\sqrt{3}T^{2}(S^{2}-1)^{3}+\left(  4S^{2}-3S^{4}+T^{2}\right)  \ \left(
\sqrt{3}+2ST+\sqrt{3}S^{2}T^{2}+2U\right)  =0\ .
\]
Since the factor $\left(  4S^{2}-3S^{4}+T^{2}\right)  $ does not vanish on our
trajectory, we have
\begin{equation}
2U=\frac{\sqrt{3}T^{2}(S^{2}-1)^{3}}{\left(  4S^{2}-3S^{4}+T^{2}\right)
}-\sqrt{3}-2ST-\sqrt{3}S^{2}T^{2}\ . \label{d6c15}
\end{equation}
Substituting the expression (\ref{d6c15}) for $U$ into either $K_{1}$ or
$K_{2}$ we find
\[
\frac{S^{2}(4-3S^{2}+T^{2})}{(4S^{2}-3S^{4}+T^{2})^{2}}\Psi=0\ ,
\]
where
\[
\Psi\!=\!4S^{2}-8T^{2}-3S^{4}+29S^{2}T^{2}-4T^{4}-22S^{4}T^{2}+14S^{2}T^{4}
+4S^{6}T^{2}-7S^{4}T^{4}+S^{2}T^{6}\ .
\]
Again, the factor $S^{2}(4-3S^{2}+T^{2})$ does not vanish, so our trajectory
is defined by (\ref{d6c15}) and a component of the curve
\[
\Psi=0\ .
\]
The leading term of $\Psi$ at $0$ is $4S^{2}-8T^{2}$ so there are two
components of the curve passing through $0$. These two components are related
by the reflection of the initial sphere, so we can, without losing generality,
take the component for which $T>0$ for $S>0$ for small $S$ and $T$.

\vskip .2cm
We are now looking at the maximum value of the square of the distance
$d_{AB}^{2}=d_{AD}^{2}=d_{BD}^{2}$ on our trajectory. The simplest way to do
this is to find the maximum value of $d_{AB}^{2}$ constrained to the curve
$\Psi=0$ since both expressions, $d_{AB}^{2}$ and $\Psi$, do not contain the
variable $U$. Moreover, only even powers of $S$ and $T$ appear in $d_{AB}^{2}$
and $\Psi$ so we set $s=S^{2}$ and $t=T^{2}$ and look for the maximum value of
the function
\[
F=\frac{12t\left(  1-s\right)  ^{2}}{(4-3s+t)(s+t)}
\]
with the constraint
\[
\psi=0\ ,\ \ \text{where}\ \ \psi=4s-8t-3s^{2}+29st-4t^{2}-22s^{2}
t+14st^{2}+4s^{3}t-7s^{2}t^{2}+st^{3}\ .
\]
Let
\begin{equation}
x=\frac{1-s}{t+1}\left(  =\cos^{2}(\varphi)\cos^{2}(\delta)\right)  .
\label{d6c16}
\end{equation}
In the variables $t$ and $x$, the expression $\psi$ has the following form:
\[
\psi=-(1+t)^{3}\left(  -1-2x+tx+3x^{2}+7tx^{2}+4tx^{3}\right)  \ .
\]
The factor $(1+t)$ is non-zero, hence the relation $\psi=0$ implies
\begin{equation}
t=\frac{1+2x-3x^{2}}{x\left(  1+7x+4x^{2}\right)  }=\frac{(1+3x)(1-x)}
{x\left(  1+7x+4x^{2}\right)  } \label{d6c17}
\end{equation}
along our component of the constraint curve. Note that the zeros of the
polynomial $1+7x+4x^{2}$ are negative while the values of the variable $x$
are, by construction, positive, see (\ref{d6c16}).

The function $F$ in the variables $t$ and $x$ reads
\[
F=\frac{12tx^{2}}{(1-x)(1+3x)}\ .
\]
Substituting the expression (\ref{d6c16}) for $t$ we find that, along our
curve,
\begin{equation}
F=\frac{12x}{1+7x+4x^{2}}\ . \label{d6c18}
\end{equation}
By construction, the variable $x$ decreases from 1 to 0 on our trajectory.
It is straightforward to find that the fraction (\ref{d6c18}) on the interval
$\left(  0,1\right)  $ attains its maximum at the point
\begin{equation}
x_{\mathfrak{m}}=\frac{1}{2} \label{d6c19}
\end{equation}
with the value
\begin{equation}
F(x_{\mathfrak{m}})=\frac{12}{11}\ . \label{d6c20}
\end{equation}
From (\ref{d6c17}) we now obtain the value of $\delta$ corresponding to this
point:
\begin{equation}
t_{\mathfrak{m}}=\tan^{2}(\delta_{\mathfrak{m}})=\frac{5}{11}\ , \label{d6c21}
\end{equation}
and then, by (\ref{d6c16}), the value of $\varphi$:
\begin{equation}
s_{\mathfrak{m}}=\sin^{2}(\varphi_{\mathfrak{m}})=\frac{3}{11}\ .
\label{d6c22}
\end{equation}
Finally, eq. (\ref{d6c15}) gives the value of $\varkappa$:
\begin{equation}
U_{\mathfrak{m}}=\tan(\varkappa_{\mathfrak{m}}-\frac{\pi}{6})=-\frac{1}
{11}\sqrt{3}\left(  4+\sqrt{5}\right)  \ , \label{d6c23}
\end{equation}
which means that
\[
\tan(\varkappa_{\mathfrak{m}})=-\frac{1}{\sqrt{15}}\ .
\]

At the point $(\varphi_{\mathfrak{m}},\varkappa_{\mathfrak{m}},\delta
_{\mathfrak{m}})$ the square of the distance in the last triplet is
\[
d_{AE}^{2}=\frac{540}{143}>\frac{12}{11}\ .
\]

The radius of the touching cylinders is given by $\left(  \ref{11}\right)  :$
\begin{equation}
r_{\mathfrak{m}}=\frac{\sqrt{\frac{12}{11}}}{2-\sqrt{\frac{12}{11}}}=\frac
{1}{8}\left(  3+\sqrt{33}\right)  \approx1.093070331\ . \label{37}
\end{equation}

Our trajectory $\gamma$ is parameterized by the variable $x$ as follows:
\begin{equation}
S=2\sqrt{\frac{(1-x)x(1+x)}{1+7x+4x^{2}}}\ , \label{traj1}
\end{equation}
\begin{equation}
T=\sqrt{\frac{(1-x)(1+3x)}{x+7x^{2}+4x^{3}}}\ , \label{traj2}
\end{equation}
\[
U=\frac{1}{2}\left(  -\sqrt{3}-\frac{4(1-x)\sqrt{(1+x)(1+3x)}}{1+7x+4x^{2}
}+\frac{\sqrt{3}(-1+5x)}{1+7x+4x^{2}}\right)  \ .
\]
The last equation can be rewritten in the form
\begin{equation}
\tan(\varkappa)=\frac{x-1}{\sqrt{(1+x)(1+3x)}}\ . \label{traj3}
\end{equation}

It is interesting to note that the point where the function $F$ gets back its
initial value 1 is also (as $x_{\mathfrak{m}}$) rational: $x=1/4$.

\section{Generalizations}

In this section we briefly consider the analogous deformation of $2n$ congruent 
parallel cylinders touching the unit ball, for values of $n$ different from 3.
We start by presenting the formulas needed and then prove that for $n>2$ the
configuration is unlockable. The case $n=2$ is special and we consider it in detail.

\subsection{Various distances}

Let $\alpha$ be the `angle' between two neighboring vertical cylinders
($\alpha$ is $\frac{\pi} {3}$ for $n=3$). Our initial configuration,
generalizing the configuration of three lines $A$, $B$ and $D$,
is
\[
C_{3} \equiv C_{3}\left(  0,0,0\right)  =\textstyle{ \left\{  \left[  \left(
0,\frac{\alpha}{2}\right)  ,\uparrow\right]  ,\left[  \left(  0,\frac{3\alpha
}{2}\right)  ,\uparrow\right]  ,\left[  \left(  0,\frac{5\alpha}{2}\right)
,\uparrow\right]  \right\}  }\ .
\]
We will study its deformations
\[
C_{3}\left(  \varphi,\delta,\varkappa\right)  =\{ A,B,D\}\ \ ,\ \text{where}
\]
\[
\textstyle{ A=\left[  \left(  \varphi,\frac{\alpha}{2}-\varkappa\right)
,\uparrow_{\delta}\right]  ,B=\left[  \left(  \varphi,\frac{5\alpha}
{2}-\varkappa\right)  ,\uparrow_{\delta}\right]  ,
D=\left[  \left(  -\varphi,\frac{3\alpha}{2}+\varkappa\right)  ,\uparrow
_{\delta}\right]  .
}
\]
For the future use we introduce the notation
\[
\gamma:=\varkappa-\frac{\alpha}{2}\ ,\ \bar{\gamma}:=\varkappa+\frac{\alpha
}{2}\ .
\]
In the coordinates $\left(  \ref{d6c3}\right)  $ we find, after lengthy
computations, that
\begin{equation}
d_{AB}^{2}=\frac{4\sin(\alpha)^{2}(1-S^{2})^{2}T^{2}}{(S^{2}+T^{2}
)(1-\sin(\alpha)^{2}S^{2}+\cos(\alpha)^{2}T^{2})}\ . \label{d6c5}
\end{equation}
Next, putting
\begin{equation}
\ U=\tan(\gamma),\ \ \bar{U}=\tan(\bar{\gamma})\ , \label{d6c3b}
\end{equation}
we get
\begin{equation}
d_{AD}^{2}=\frac{4(ST+U)^{2}}{1-S^{2}+T^{2}+U^{2}+2STU}\ , \label{d6c2}
\end{equation}
while
\[
d_{BD}^{2}=d_{AD}^{2}|_{\gamma\rightarrow\bar{\gamma},\delta\rightarrow
-\delta}\ .
\]
Again, the initial trigonometric formulas for these distances involve several
different trigonometric functions for each angle $\varphi,\varkappa,\delta.$
The advantage of the formulas above is that every variable $\varphi
,\varkappa,\delta$ enters each distance only via a single trigonometric
function, and so these expressions become algebraic, which permits us to write
down the final formulas.

\subsection{When can the distances $d_{AB},$ $d_{AD},$ $d_{BD}$ grow?}

Here we look for the range of $\alpha$ in which all the distances $d_{AB},$
$d_{AD},$ $d_{BD}$ increase above the value $4\sin(\alpha/2)^{2}$ (the initial
distance between the lines $A$ and $D$) as we move away from the point
$C_{3}\left(  0,0,0\right)  $ in $C_{3}\left(  \varphi,\delta,\varkappa
\right)  .$ As a result, we will prove the following statement.

\begin{proposition}
For any $n=2k\geq6$ the configuration of $n$ congruent parallel non-intersecting
cylinders, touching the unit ball, can be unlocked.
\end{proposition}

\begin{proof}
Let us consider the curve
\begin{equation}
\varphi=\sum_{j>0}\varphi_{j}t^{j}\ ,\ \delta=\sum_{j>0}\delta_{j}
t^{j}\ ,\ \varkappa=\sum_{j>0}\varkappa_{j}t^{j}\ , \label{d6c6}
\end{equation}
and study the expansions of the distances $d_{\ast\ast}^{2}$ in $t.$ The
coefficient in $t^{k}$, $k=0,1,...$,
is denoted by $\left[  d_{\ast\ast}^{2}\right]  _{k}$. We have
\[
\left[  d_{AB}^{2}\right]  _{0}=\frac{4\delta_{1}^{2}\sin(\alpha)^{2}}
{\delta_{1}^{2}+\varphi_{1}^{2}}\ .
\]
This is greater than or equal to $4\sin(\alpha/2)^{2}$
iff
\begin{equation}
\delta_{1}^{2}\sin(\alpha)^{2}\geq(\delta_{1}^{2}+\varphi_{1}^{2})\sin
(\alpha/2)^{2}\ . \label{d6c7}
\end{equation}
Next,
\[
\left[  d_{AD}^{2}\right]  _{0}=\left[  d_{BD}^{2}\right]  _{0}=4\sin
(\alpha/2)^{2}
\]
and
\[
\left[  d_{AD}^{2}\right]  _{1}=-4\varkappa_{1}\sin(\alpha)\ ,\ \left[
d_{BD}^{2}\right]  _{1}=4\varkappa_{1}\sin(\alpha)\ .
\]
For both distances to weakly grow, we have to set $\varkappa_{1}=0$. Then
\[
\left[  d_{AD}^{2}\right]  _{2}=-\sin(\alpha)\left(  2\delta_{1}\varphi
_{1}+4\varkappa_{2}+2\delta_{1}\varphi_{1}\cos(\alpha)+\sin(\alpha)(\delta
_{1}^{2}-\varphi_{1}^{2})\right)  \ ,
\]
\[
\left[  d_{BD}^{2}\right]  _{2}=\sin(\alpha)\left(  2\delta_{1}\varphi
_{1}+4\varkappa_{2}+2\delta_{1}\varphi_{1}\cos(\alpha)-\sin(\alpha)(\delta
_{1}^{2}-\varphi_{1}^{2})\right)  \ .
\]
Both of these are positive ($\sin(\alpha)>0$) iff
\[
-2\delta_{1}\varphi_{1}(1+\cos(\alpha))+(\delta_{1}^{2}-\varphi_{1}^{2}
)\sin(\alpha)\leq4\varkappa_{2}\
\]
and
\[
4\varkappa_{2}\leq-2\delta_{1}\varphi_{1}(1+\cos(\alpha))-(\delta_{1}
^{2}-\varphi_{1}^{2})\sin(\alpha)\ .
\]
This can be solved for $\varkappa_{2}$ if
\[
-2\delta_{1}\varphi_{1}(1+\cos(\alpha))+(\delta_{1}^{2}-\varphi_{1}^{2}
)\sin(\alpha)\leq-2\delta_{1}\varphi_{1}(1+\cos(\alpha))-(\delta_{1}
^{2}-\varphi_{1}^{2})\sin(\alpha)
\]
or
\begin{equation}
\delta_{1}^{2}\leq\varphi_{1}^{2}\ . \label{d6c8}
\end{equation}
The inequalities (\ref{d6c7}) and (\ref{d6c8}) are compatible iff
\[
\cos(\alpha/2)^{2}\geq1/2\ \ \text{or}\ \ \alpha\leq\pi/2\ .
\]

For $\alpha<\pi/2$ this analysis is sufficient to show that there is room
for the cylinders $A$, $B$ and $D$ to grow. It is not too difficult to see
that the other values of distances between pairs of cylinders in our
configuration of $2n$ cylinders are not relevant. Thus, our Proposition is
proved. In particular, this analysis proves also the infinitesimal version of
our Theorem \ref{Main}.
\end{proof}

\vskip .3cm 
For $\alpha=\pi/2$ (the case of four cylinders) we find
$\delta_{1} ^{2}=\varphi_{1}^{2}$ and $4\varkappa_{2}=-2\delta_{1}\varphi_{1}
$, so further analysis is needed. It turns out that there is only one possible
motion here: the cylinders $A$ and $D$ stay parallel, the remaining two stay
parallel as well, so, up to a global rotation of all four cylinders, one
parallel pair is fixed while the other one rotates. We will show this in
Subsection \ref{alpi4}.

\vskip .2cm
Also, we have analyzed another strategy of unlocking, when the family of
possible motions $C_{3}\left(  \varphi,\delta,\varkappa\right)  $ is replaced
by $\bar{C}_{3}\left(  \varphi,\delta,\varkappa\right)  =\{ A,B,D\}$ where
\[
\textstyle{ A=\left[  \left(  \varphi,\frac{\alpha}{2}-\varkappa\right)
,\uparrow_{\delta}\right]  ,B=\left[  \left(  \varphi,\frac{5\alpha}
{2}-\varkappa\right)  ,\uparrow_{\delta}\right]  , D=\left[  \left(
-\varphi,\frac{3\alpha}{2}+\varkappa\right)  ,\uparrow_{-\delta}\right]  \ . }
\]
The difference here is the change of $\delta$ to $-\delta$ for the cylinder
$D$.
This other strategy corresponds to a different embedding of the symmetry group
$\mathbb{D}_{3}$ (in case of $n=3$, i.e. $\alpha=\frac{\pi}{3}$) in $O\left(
3\right)  .$

\vskip .2cm
Obviously, the function $d_{AB}^{2}$ stays the same, while
\[
d_{AD}^{2}=\frac{4U^{2}S^{2}(1+T^{2})^{2}}{(S^{2}+T^{2})(1-S^{2}+U^{2}
+T^{2}U^{2})},\ \ d_{BD}^{2}=d_{AD}^{2}|_{\gamma\rightarrow\bar{\gamma}
,\delta\rightarrow-\delta}.
\]
For the curve (\ref{d6c6}), we have
\[
\left[  d_{AD}^{2}\right]  _{0}=\left[  d_{BD}^{2}\right]  _{0}=\frac
{4\varphi_{1}^{2}\sin(\alpha/2)^{2}}{\delta_{1}^{2}+\varphi_{1}^{2}}\ .
\]
This is greater than or equal to $4\sin(\alpha/2)^{2}$ if $\varphi_{1}^{2}
\geq\delta_{1}^{2}+\varphi_{1}^{2}$ which implies $\delta_{1}=0$. Then
(\ref{d6c7}) implies $\varphi_{1}=0$ so for $0<\alpha<\pi$ the cylinders
cannot be unlocked using this other strategy.

\subsection{Four cylinders \label{alpi4}}

The initial position of four cylinders is shown on Figure
\ref{four cylinders initial postion}.

\begin{figure}[th]
\centering
\includegraphics[scale=0.226]{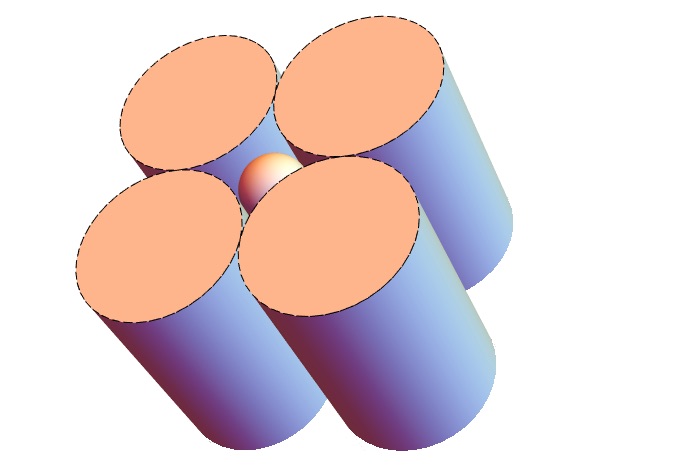} \caption{Initial
position}
\label{four cylinders initial postion}
\end{figure}

\begin{proposition}
The configuration of four parallel non-intersecting cylinders of radius
$r=1+\sqrt{2}$ touching the unit ball, being not rigid, is not unlockable in
our regime.
\end{proposition}

\begin{proof}
For $n=4,$ the angle $\alpha=\frac{\pi}{2},$ so $\bar{U}=-1/U$, see $\left(
\ref{d6c3b}\right)  ,$ and
\[
d_{AB}^{2}=\frac{4\left(  1-S^{2}\right)  T^{2}}{S^{2}+T^{2}}\ ,
\]
\[
d_{AD}^{2}\!=\!\frac{4(ST+U)^{2}}{1-S^{2}+T^{2}+2STU+U^{2}}\ ,\ d_{BD}^{2}
\!=\!\frac{4(-1+STU)^{2}}{1-2STU+U^{2}-S^{2}U^{2}+T^{2}U^{2}}\ .
\]
Let $Q:=T^{2}-S^{2}-2S^{2}T^{2}.$ The system of inequalities
\[
d_{AB}^{2}\geq2\ ,
\]
\[
d_{AD}^{2}\geq2\ ,\ d_{BD}^{2}\geq2
\]
is equivalent to
\begin{equation}
Q\geq0\ , \label{d6c9}
\end{equation}
\begin{equation}
U^{2}+2STU\geq1+Q\ ,\ 1-2STU\geq U^{2}(1+Q)\ . \label{d6c10}
\end{equation}
The sum of last two inequalities is
\[
U^{2}+1\geq(U^{2}+1)(1+Q),\ \ \text{hence}\ \ 1+Q\leq1\ ,
\]
which, along with (\ref{d6c9}) gives
\begin{equation}
Q=0\ \ ,\ \ \text{or}\ \ S^{2}=\frac{T^{2}}{1+2T^{2}}\ . \label{d6c11}
\end{equation}
Now the inequalities (\ref{d6c10}) become
\[
U^{2}+2STU\geq1\ ,\ 1-2STU\geq U^{2},
\]
therefore
\begin{equation}
U^{2}+2STU=1\ . \label{d6c12}
\end{equation}
Eqs. (\ref{d6c11}) and (\ref{d6c12}) define uniquely the trajectory, depicted
in Figure \ref{four cylinders motion}.
\end{proof}

\begin{figure}[th]
\centering
\includegraphics[scale=0.226]{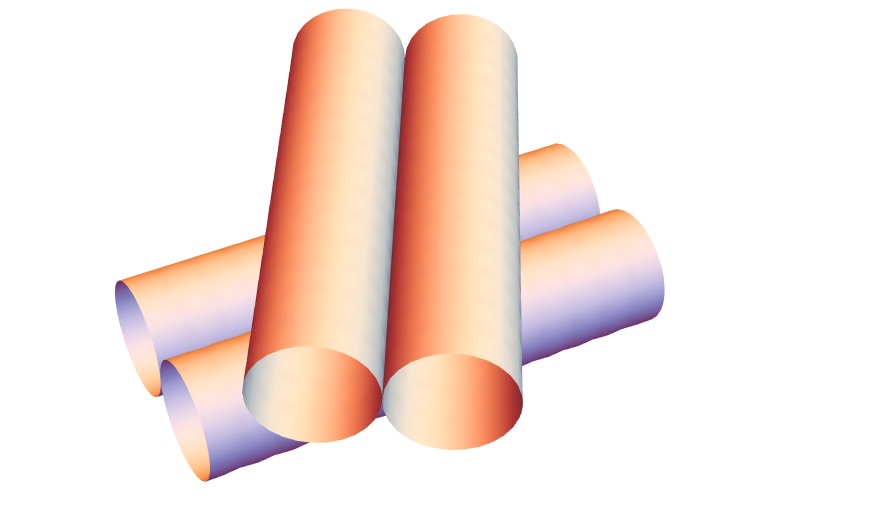}\caption{Motion of four
cylinders}
\label{four cylinders motion}
\end{figure}

\vskip .2cm \noindent\textbf{Conjecture.} We believe that these are all
possible positions of four cylinders of radius $r=1+\sqrt{2}$ touching the
unit ball.

\vskip.2cm If that would be the case, then, obviously, one could not put 5
non-intersecting cylinders of radius $r=1+\sqrt{2}$ in contact with unit ball,
thus answering the analogue of the initial $n=6$ question of Kuperberg. But
this last statement can be proven independently of the above conjecture.

\begin{proposition}
It is not possible to place five non-intersecting cylinders of radius
$r=1+\sqrt{2}$ in such a way that all of them touch a unit ball.
\end{proposition}

\begin{proof}
Suppose the opposite. Consider the corresponding configuration of 5 cylinders.
Let us inscribe into them 5 balls of the same radius $r=1+\sqrt{2},$ each
touching the central unit ball. As we will explain in the next paragraphs (see also
\cite{KKLS}), any configuration of five non-intersecting balls of
radius $r=1+\sqrt{2}$ touching the central unit ball contains a triple, which,
$\operatorname{mod}SO\left(  3\right)$, is formed by a ball on the North
pole, a ball on the South pole and a ball on the intersection of Greenwich and
the equator. The three non-intersecting cylinders of radius $r=1+\sqrt{2}$, containing
these three balls, have to be parallel (and perpendicular to Greenwich), which
leaves a uniquely defined place for just one more cylinder.

\vskip .2cm
In order to understand the configuration manifold of five non-intersecting
balls of radius $r=1+\sqrt{2}$ touching the central unit ball, let us position
one of them at the South. Consider the set {\sf T} of the three balls closest to this
$\mathsf{S}$ ball. Their centers lie in the (closed) northern hemisphere.

\vskip .2cm
Consider first the case when no ball from {\sf T} touches the two others. If
at least one of them has its center not on the equatorial plane, then there is
no room left for the fifth ball. Therefore all three centers must be on the equator
plane, and then the fifth ball is fixed to be the $\mathsf{N}$ ball. Our three
equatorial balls are then free to use the equatorial plane. (This shows that
the dimension of the configuration manifold of our 5 balls is two,
$\operatorname{mod}SO\left(  3\right)  $.)  The $\mathsf{N}$ ball, the
$\mathsf{S}$ ball and any one from the equatorial balls make then the triple sought.

\vskip .2cm
In the remaining case, when one ball from {\sf T} touches the other two,
the triple itself forms a configuration of the type needed.
\end{proof}

\vskip .4cm
\noindent{\bf Remark.} Firsching \cite{F}, building on the work \cite{BW}, proves a stronger result:
five disjoint infinitely long cylinders with radius $r>1.89395$ cannot touch a unit ball.

\section{Conclusion}

In this paper we were attempting to understand better the question of W.
Kuperberg about the maximum number of non-intersecting equal (infinite)
cylinders of radius $r\geq1$ touching the unit ball in $\mathbb{R}^{3}.$ The
open conjecture is that this number is 6.
We were able to clarify a related question of how large the radius $r$ of six
cylinders can be in order that the non-intersection condition is satisfied.

\vskip .2cm
We believe that the \textit{record} configuration $C_{6}(\varphi_{\mathfrak{m}},\varkappa_{\mathfrak{m}},\delta_{\mathfrak{m}})$ we found,
which has all the relevant distances equal to $\sqrt{\frac{12}{11}},$ gives 
the best possible value for $r$, see $\left(  \ref{37}\right)  .$

\vskip .2cm
In the forthcoming papers \cite{OS,OS2} we shall investigate the local maximality properties of several configurations of six cylinders.

\vskip .3cm
It is interesting to note that all the angles describing the configuration
$C_{6}(\varphi_{\mathfrak{m}},\varkappa_{\mathfrak{m}},\delta_{\mathfrak{m}})$
are pure geodetic, in the sense of \cite{CRS}: an angle $\alpha$ is pure
geodetic if the square of its sine
is rational. Formally it is explained as follows: for
any rational $x$ the formulas (\ref{traj1}), (\ref{traj2}) and (\ref{traj3})
define pure geodetic angles and our record configuration is attained at
$x=1/2$.

\vskip .3cm\noindent{\footnotesize {\textbf{Acknowledgements.} Part of the
work of S. S. has been carried out in the framework of the Labex Archimede
(ANR-11-LABX-0033) and of the A*MIDEX project (ANR-11-IDEX-0001-02), funded by
the ``Investissements d'Avenir" French Government programme managed by the
French National Research Agency (ANR). Part of the work of S. S. has been
carried out at IITP RAS. The support of Russian Foundation for Sciences
(project No. 14-50-00150) is gratefully acknowledged by S. S. The work of O.
O. was supported by the Program of Competitive Growth of Kazan Federal
University and by the grant RFBR 17-01-00585.}}

\end{document}